\documentclass{article}
\usepackage{amsmath,amssymb}
\usepackage{theorem}

\newtheorem{theorem}{Theorem}
\newtheorem{proposition}{Proposition}

\newtheorem{corollary}{Corollary}
\newtheorem{lemma}{Lemma}

{\theorembodyfont{\rmfamily} }
\newenvironment{proof}{\noindent\textsc{Proof}}{\hfill\ensuremath{\blacksquare}}

\newcommand{\rad}{\textup{rad}}
\newcommand{\Z}{\mathbb{Z}}

\newcommand{\Size}[1]{\left|#1\right|}
\newcommand{\Group}[1]{\left<#1\right>}
\newcommand{\Fix}{\mathrm{Fix}}

\begin{document}

\title{The  $p$-Modular Descent Algebras}
\author{M.D. Atkinson\\S.J. van Willigenburg\\
School of Mathematical and Computational Sciences\\
North Haugh, St Andrews, Fife KY16 9SS, UK\\ \\
G. Pfeiffer\\Department of Mathematics\\University College, Galway, Eire}
\maketitle

\begin{abstract} Solomon's descent algebras are studied over fields
of prime characteristic.  Their radical and irreducible modules are
determined.  It is shown how their representation theory can be
related to the representation theory in fields of characteristic zero.
\end{abstract}

\section{Introduction}\label{Intro}
Descent algebras are non-commutative, non-semi-simple algebras
associated with Coxeter groups.  They were first discovered by Solomon
in the 1970's and for the last 10 years have been studied
intensively.  Previous work has concentrated on the case that the
underlying field has characteristic zero.  However, as we shall soon
see, characteristic $p$ analogues exist and their structure is very
sensitive to the value of $p$.  This paper determines the radical of
a descent algebra in characteristic $p$, and the irreducible
modules.  It also explains how the representation theory is connected
to the representation theory in characteristic zero.

Let $W$ be a Coxeter group with generating set $S$ of fundamental
reflections. Thus every element $w\in W$ can be written as a product
of elements in $S$; we let $\lambda(w)$ denote the length of a
shortest expression for $w$.  If $L$ is any subset of $S$ let
$W_L$ be the subgroup generated by $L$.  $W_{L}$ is called a
{\em standard parabolic} subgroup of $W$ and any subgroup conjugate to
a standard parabolic subgroup is said to be parabolic.
Let $X_L$ be the (unique) set of
minimal length representatives of the left cosets
of $W_L$ in $W$. Notice that $X_L^{-1}=\{g^{-1}|g\in X_L\}$ is
then a set of
representatives (also of minimal length) for the right cosets of
$W_L$ and that $X_K^{-1} \cap X_L$ is a set representatives for the
double cosets corresponding to $W_{L},W_{K}$.

Solomon proved the following remarkable theorem:
\begin{theorem}\cite{solomon-mackey}
\label{so-lomon1}
For every subset $K$ of $S$ let
\[x_K=\sum_{w\in X_K}w.\]
Then
\[x_Jx_K=\sum a_{JKL}x_L\]
where $a_{JKL}$ is the number of elements $g\in X_{J}^{-1}\cap X_K$ such that
$g^{-1}W_Jg\cap W_K=W_L$ with $L=g^{-1}Jg\cap K$.
\end{theorem}

The set of all $x_K$ is therefore a basis for an algebra $\Sigma _W$ over
the field of
rationals with integer structure constants $a_{JKL}$.  This algebra
is now known as the {\em descent algebra} of $W$ and much is
known about its structure
\cite{Atk,A&vW,B&B-maple,B&B,Ber,BBH&T,Cel,G&R,GKLL&T}.

Solomon himself began the study of this algebra by determining its
radical, $\rad(\Sigma_{W})$, and some properties of
$\Sigma_{W}/\rad(\Sigma_{W})$.  To describe his results let $\chi_{K}$
be the permutation character of $W$ acting on the right cosets of
$W_{K}$ and let $G_{W}$ be the $Z$-module generated by all
$\chi_{K}$.  Note that each generalised character in $G_{W}$ has
integer values on the elements of $W$.

\begin{theorem} \cite{solomon-mackey}
\label{so-lomon2}
\begin{enumerate}
\item $\rad(\Sigma_{W})$ is spanned by all differences $x_{J}-x_{K}$
where $J$ and $K$ are conjugate subsets of $S$
\item the linear map $\theta$ defined by the images
$\theta(x_{K})=\chi_{K}$ is an algebra homomorphism, and
$\ker\theta=\rad(\Sigma_{W})$
\end{enumerate}
\end{theorem}

Since the structure constants $a_{JKL}$ are integers the $Z$-module
${\cal Z}_W$ spanned by all $x_{K}$ is a subring (an order) of
$\Sigma_{W}$.  This allows us to study the $p$-modular version of the
descent algebra for any prime $p$.  For any prime $p$, $p{\cal Z}_W$
 is an ideal of ${\cal Z}_W$.
We define $\Sigma (W,p)={\cal Z}_W/p{\cal Z}_W$, the $p$-modular descent
algebra
of  $W$.
Obviously, $\Sigma (W,p)$ is an algebra over ${\cal F}_{p}$, the field of
order
$p$.

Let  $\rho_1$ be the natural projection ${\cal Z}_W\rightarrow
\Sigma(W,p)$
and let $\overline{x}_J=\rho_{1}(x_{J})$.  Then
$$\overline{x}_J\overline{x}_K=\sum\overline{a}_{JKL}\overline{x}_L$$
where $\overline{a}_{JKL}$ is the image of $a_{JKL}$ in ${\cal F}_{p}$.
Furthermore let $\rho_2$ be the map defined on $G_{W}$ which reduces
character values modulo $p$, and let $G(W,p)$ be the image of
$\rho_{2}$.

The map $\phi:\Sigma (W,p)\rightarrow G(W,p)$ defined by
$$ \phi(\rho_1(x))=\rho_2(\theta(x)) \mbox{ for all }x\in{\cal Z}_W$$
is clearly well-defined and is an algebra homomorphism.  In
section \ref{radical-section} we shall give an analogue of Solomon's
theorem to describe
the radical of $\Sigma(W,p)$ using the homomorphism $\phi$.  In
section \ref{representation-section} we build on
this result by defining the irreducible modules
of $\Sigma(W,p)$.  Then we relate the representation theory of
$\Sigma(W,p)$ to
that of $\Sigma_{W}$ and give explicit details for each of
the Coxeter types.  We begin, in section \ref{marks-section}, by
introducing a tool that we use throughout the paper: the parabolic
table of marks.  Many of our results were first suggested by computation
with GAP \cite{GAP}.

\section{The parabolic table of marks}\label{marks-section}

We
first recall the definition and basic  properties of the Burnside ring
(see \cite[chapter~11]{CR2}  for details) and the table  of marks of a
finite group.    In  the   case   of  a  finite Coxeter   group    the
\emph{parabolic\/} table of marks is defined as a certain submatrix of
the table of marks.

Let $G$ be a finite group and let $G_1  (=1), G_2, \dots, G_r (=G)$ be
representatives  of the conjugacy   classes of subgroups of  $G$. Then
each $G$-set decomposes as a disjoint union of transitive $G$-sets and
each transitive $G$-set is isomorphic to $G/G_i$ for some $i$.  Denote
by $[X]$ the isomorphism type  of the $G$-set $X$.  The \emph{Burnside
  ring\/} $\Omega(G)$   is    the  ring  of  formal    integer  linear
combinations of isomorphism types of $G$-sets with addition defined by
disjoint unions of $G$-sets
\[
  [X] + [Y] = [X \dot{\cup} Y]
\]
and multiplication defined by the cartesian product
\[
   [X] \cdot [Y] = [X \times Y].
\]
The transitive $G$-sets $G/G_i$ form a basis of $\Omega(G)$ so
\[
  \Omega(G) = \Bigl\{\sum_{i=1}^r a_i [G/G_i] \Bigm| a_i \in \Z \Bigr\}
\]
the  free  abelian  group   generated  by  the  isomorphism  types  of
transitive $G$-sets.

The \emph{table of marks\/} of $G$ is the $r \times r$-matrix
\[
  M(G) = \left(\Size{\Fix_{G/G_i}(G_j)}\right)_{i,j=1,\dots,r}
\]
which records for  subgroups $G_i$, $G_j$ of  $G$ the number of  fixed
points of   $G_j$ in the  action  of $G$ on  the  cosets of $G_i$ (the
\emph{mark\/} of $G_j$ on $G/G_i$), and where both $G_i$ and $G_j$ run
through  the (system  of)   representatives  of conjugacy  classes  of
subgroups of $G$.

We have
\[
  \Size{\Fix_{G/G_i}(G_j)} = \Size{N_G(G_i) : G_i} \cdot
  \Size{\{G_i^x \mid x \in G,\, G_j \leq G_i^x\}}
\]
Thus, with  a suitable ordering  of the representatives $G_i$,  we see
that $M(G)$ is a lower triangular matrix  with non-zero entries on the
diagonal, and therefore invertible.

Each finite $G$-set $X$ has an associated vector of fixed point numbers
\[
  \beta_X = ( \Size{\Fix_{X}(G_j)} )_{j=1}^r
\]
and if $[X] = \sum a_i [G/G_i]$ in $\Omega(G)$ then $\beta_X = (\dotsc
a_i  \dotsc) M(G)$.  Disjoint union and  cartesian product of $G$-sets
translate  into  componentwise  addition  and multiplication  of fixed
point vectors.  We thus have
\begin{theorem}[Burnside 1911]
  The map $\beta \colon \Omega(G) \to \Z^r, \; [X] \mapsto \beta_X$ is
  a well defined injective homomorphism of  rings.  In particular, $X$
  and $Y$  are  isomorphic as  $G$-sets if   and only  if  $\beta_X  =
  \beta_Y$.
\end{theorem}

Now  let $G= (W,  S)$ be a  finite Coxeter group and let $E$ be a set
of representatives of conjugate subsets of $S$.  The intersection of
any two    parabolic subgroups  of   $W$  is a  parabolic   subgroup.
Therefore, if  $J,K \subseteq  S$, the  direct product   $W/W_J \times
W/W_K$ decomposes as a sum  of transitive $W$-sets,  each of which  is
isomorphic to $W/W_L$ for some $L \subseteq S$.  Thus the coset spaces
$W/W_J$, where  $J$ runs through the set $E$, form the basis of a subring
\[
  \Omega^c(W) = \left<[W/W_J] \mid J \subseteq S\right>
  = \Bigl\{ \sum_{J \in E} \alpha_J [W/W_J] \Bigm| \alpha_J \in \Z\Bigr\}
\]
of $\Omega(W)$, the \emph{parabolic Burnside ring\/} of $(W, S)$,
first introduced in \cite{BBH&T}.

We call the   corresponding part of  the table  of  marks of $W$   the
\emph{parabolic table of marks\/} of $W$, and denote it by
\[
  M^c(W) = \left(\Size{\Fix_{W/W_J}(W_K)}\right)_{J,K\in E}
\]
where we also  write  $\beta_{JK} = \Size{\Fix_{W/W_J}(W_K)}$  for any
$J,K \subseteq S$.  Note that by the formula above  for
$\Size{\Fix_{G/G_i}(G_j)}$ we have the following result \cite{BBH&T}.
\begin{lemma}\label{betaJK=aJKK}
\begin{multline*}
  \beta_{JK} = \Size{N_W(W_J) : W_J} \cdot
  \Size{\{W_J^w \mid w \in W,\, W_K \leq W_J^w\}} \\
  = \Size{\{w \in X_J^{-1}\cap X_K \mid J^w \cap K = K\}} = a_{JKK}.
\end{multline*}
In  particular,  $\beta_{JJ}  =  \Size{N_W(W_J)   :  W_J} \not=0$  and
$\beta_{JJ}$ divides $\beta_{JK}$ for every $K \subseteq S$.
\end{lemma}
We have that the map $\beta^c  \colon \Omega^c(W) \to \Z^E, \; [W/W_J]
\mapsto (\beta_{JK})_{K \in E}$  is a well defined injective
ring homomorphism.

The parabolic marks of an arbitrary subgroup of  $W$ coincide with the
marks of a particular parabolic subgroup associated to it.  Let $U \leq
W$ and define the \emph{parabolic closure\/} $U^c$ of $U$ in $W$ as
\[
  U^c = \bigcap \left\{W_J^w \mid
                       J \subseteq S, w \in W, U \leq W_J^w\right\},
\]
the intersection of  all parabolic subgroups of $W$  that contain $U$.
Then $U^c$ is a parabolic subgroup and conjugate to  $W_K$ for some $K
\in E$.

\begin{proposition}
  Let  $U \leq W$   and $K \subseteq S$.  Then   $U^c$ is conjugate to
  $W_K$  if and only  if $\Size{\Fix_{W/W_J}(U)} = \beta_{JK}$ for all
  $J \subseteq S$.
\end{proposition}

\begin{proof}
  We have $\Fix_{W/W_J}(U) \supseteq \Fix_{W/W_J}(U^c)$ since $U  \subseteq
U^c$.  Now
  let $x \in \Fix_{W/W_J}(U)$.  We  may assume  that  $x = W_J$.   But
  then, $U  \leq W_J$, and by the  definition of $U^c$, also $U^c \leq
  W_J$, whence $x \in \Fix_{W/W_J}(U^c) = \Fix_{W/W_J}(U)$.

  The  converse follows from the fact  that $U^c$ is conjugate to some
  $W_K$, and that all the columns of the  parabolic table of marks are
  different.
\end{proof}

As a corollary we obtain complete information about  the values of the
permutation characters  afforded  by the $W$-sets $W/W_J$.   Note that
$\left<c_K\right>^c =   W_K$ for a Coxeter  element   $c_K$ of  $W_K$.
Furthermore note that a transitive permutation character value coincides
with the
mark of a cyclic subgroup, both being the same number of fixed points.

\begin{corollary}
  For $J  \subseteq S$ let  $\chi_{J}$ be, as in section \ref{Intro},
   the permutation character of
  $W$ on $W/W_J$.   Then $\chi_{J}(c_K) = \beta_{JK}$.   Moreover, for
  $w \in W$ we have  $\chi_{J}(w) = \beta_{JK}$ if  $\left<w\right>^c$
  is conjugate to $W_K$.
\end{corollary}

\section{The radical of $\Sigma(W,p)$}\label{radical-section}

The main aim of this section is to prove the following $p$-modular
analogue of Theorem \ref{so-lomon2}

 \begin{theorem}
 \label{ma-in2}
 $\rad(\Sigma(W,p))=\ker\phi$. Moreover, $\rad(\Sigma(W,p))$ is spanned by all
 $\overline{x}_J-\overline{x}_K$ where $J,K$ are conjugate subsets
 of $S$, together with all $\overline{x}_J$ for which $p$ divides
 $[N_{W}(W_{J}):W_{J}]$.
 \end{theorem}

Let $r$ be the number of rows of $M^c(W)$ and let $s$ be the
number of
rows indexed by subsets $J$ with $p\not|[N_{W}(W_{J}):W_{J}]$.

\begin{lemma}
\label{rank}
\begin{enumerate}
\item $M^c(W)$ is a lower triangular matrix of rank $r=\dim G_{W}$
\item The $p$-rank of $M^c(W)$ (i.e. the rank of $M^c(W)$ modulo $p$ or
$\dim G(W,p)$) is $s$
\end{enumerate}
\end{lemma}
\begin{proof}
The first part follows from Section \ref{marks-section}.
If $p$ divides a diagonal entry of $M^c(W)$ then, by Lemma
~\ref{betaJK=aJKK},
$p$ divides every entry of that row.  Thus the
rank of $M^c(W)\mod p$ (i.e. $\dim G(W,p)$) is the number of non-zero rows
in $M^c(W) \mod p$ and this, by Lemma ~\ref{betaJK=aJKK} again, is $s$.
\end{proof}

\begin{lemma}
\label{co-mutt}
\begin{enumerate}
\item $\Sigma(W,p)/\rad(\Sigma(W,p))$ is commutative.
\item Every nilpotent element of $\Sigma(W,p)$ lies in $\rad(\Sigma(W,p))$
\end{enumerate}
\end{lemma}
\begin{proof}
Let $\theta_{1}$ be the restriction of $\theta$ to ${\cal
Z}_{W}$.  Then $\theta_{1}$ maps ${\cal Z}_{W}$ onto the commutative
ring $G_{W}$.  By Theorem \ref{so-lomon2}, the kernel of $\theta_{1}$ is
the $Z$-module ${\cal R}_{W}$
spanned by all $x_{J}-x_{K}$ where
$J$ and $K$ are conjugate subsets of $S$, and is a nilpotent ideal of
${\cal Z}_{W}$.  In particular \( \rho_1({\cal R}_W) \) is a nilpotent
ideal of \( \Sigma(W,p) \), and
therefore \( \rho_1({\cal R}_W)\subseteq\rad(\Sigma(W,p)) \).  Hence there
exists
an ideal \( {\cal S}_W \) of \( \Sigma_W \), the pre-image of \(
\rad(\Sigma(W,p)) \), such that \( {\cal R}_W\subseteq {\cal S}_W \) and
${\cal
S}_W/{\cal P}_W \cong \rad(\Sigma(W,p))$.  Since \( \Sigma(W,p) \cong
{\cal Z}_W/{\cal P}_W \), \(\Sigma(W,p)/\rad(\Sigma(W,p)) \cong {\cal
Z}_W/{\cal S}_W \) is a homomorphic image of \( {\cal Z}_W/{\cal R}_W
\cong G_W \).  Since the latter ring is commutative the first part
follows.

If $x$ is any nilpotent element of $\Sigma(W,p)$ then the coset
$x+\rad (\Sigma(W,p))$ is a nilpotent element in the commutative
semi-simple algebra $\Sigma(W,p)/\rad(\Sigma(W,p))$ and so is zero.
Therefore $x\in\rad(\Sigma(W,p))$ proving the second part.
\end{proof}

\begin{proof} of Theorem ~\ref{ma-in2}.

First we note that $\rad(\Sigma(W,p))\subseteq\ker\phi$.  This
is because the image of $\phi$ is a space of functions defined over a
field and
is therefore semi-simple. Consequently the two-sided nilpotent ideal
$\phi(\rad(\Sigma(W,p)))$ must be zero.

Now we prove that, if $p|[N_{W}(W_{J}):W_{J}]$, then
$\overline{x}_{J}\in\rad(\Sigma(W,p))$.  From the definition of
$a_{JKL}$ in Theorem ~\ref{so-lomon1}, $\overline{a}_{JKL}=0$ unless
$L\subseteq
K$ and, by Lemma ~\ref{betaJK=aJKK}, $\overline{a}_{JKK}=0$ also.  Thus
$\overline{x}_{J}\overline{x}_{K}$ is a linear combination of
elements $\overline{x}_{L}$ with $L\subset K$ (and so
$|L|\leq|K|-1$).  Now, by induction, it follows that
$\overline{x}_{J}^{t}\overline{x}_{K}$ is a linear combination of
elements $\overline{x}_{L}$ with $|L|\leq|K|-t$ and so
$\overline{x}_{J}^{|K|+1}\overline{x}_{K}$=0 for all $K$.  In
particular $\overline{x}_{J}$ is nilpotent and so
$\overline{x}_{J}\in\rad(\Sigma(W,p))$ by Lemma ~\ref{co-mutt}.

The elements $\overline{x}_{J}-\overline{x}_{K}$ where $J$ and $K$ are
conjugate subsets of $S$ are all nilpotent and, by Lemma
~\ref{co-mutt}, lie in $\rad(\Sigma(W,p))$.  They span a space $U$ of
dimension $\dim\rad(\Sigma_{W})=\dim\Sigma_{W}-\dim
G_{W}=2^{n-1}-r$.  In addition there are $r-s$ elements
$\overline{x}_{J}$ corresponding to those rows of $M^c(W)$ for which
$p|[N_{W}(W_{J}):W_{J}]$ which also lie in $\rad(\Sigma(W,p))$.
These, together with $U$, span a space of dimension
$2^{n-1}-r+(r-s)=2^{n-1}-\dim G(W,p)=\dim \ker\phi$.  Hence
$\dim\rad(\Sigma(W,p))\geq\dim\ker\phi$.

This proves that $\ker\phi=\rad(\Sigma(W,p))$ as required and that it
is spanned by the desired set of elements.
\end{proof}

\section{Representation Theory of
$\Sigma(W,p)$}\label{representation-section}

The representation theory of $\Sigma_{W}$ has not been much studied in
general although some results for the Coxeter groups of types $A$ and
$B$ have been found \cite{A&vW,B&B,Ber,G&R}.  In this section we show how
the representation theory of $\Sigma(W,p)$ depends on that
of $\Sigma_{W}$.  Specifically, we shall be interested in the
composition factors of the principal indecomposable modules
(indecomposable summands of the regular module) for each of $\Sigma_{W}$
and $\Sigma(W,p)$.  The first observation is straightforward: a
representation of $\Sigma_{W}$ over ${\cal F}_{p}$ necessarily has
$p{\cal Z}_W$ in its kernel
and so induces a representation of $\Sigma(W,p)$; moreover, every
representation of $\Sigma(W,p)$ arises in this way.  Therefore we may
study the representation theory of $\Sigma(W,p)$ by examining the
$p$-modular representations of $\Sigma_{W}$.  We do this in the manner
pioneered in group theory: by relating the representations in
characteristic zero to those in characteristic $p$ via a
decomposition matrix.

This approach is tractable because the irreducible representations
are all $1$-dimensional.  In fact, since $\Sigma_{W}/\rad(\Sigma_{W})$
and $\Sigma(W,p)/\rad(\Sigma(W,p))$ are commutative of dimensions $r$
and $s$ respectively (where $r$ and $s$ have the meanings given in
the previous section) $\Sigma_{W}$ has $r$ $1$-dimensional
irreducible representations over a field of characteristic zero and $s$
$1$-dimensional irreducible representations
over a field of characteristic $p$.  It follows (see 54.16,
\cite{C&R}) that the multiplicities of the principal indecomposable
modules as direct summands
in the regular representation of both $\Sigma_{W}$ and $\Sigma(W,p)$
are all 1.

We can explicitly describe the irreducible representations.  As in
Section \ref{marks-section} let $E$
denote the set of representatives of the subsets of $S$
 that index the rows and columns of $M^c(W)$.  For
each $K\in E$ define the map $\lambda_{K}:\Sigma_{W}\rightarrow \mathbb{Q}$
 by
$$\lambda_{K}(x)=\theta(x)(c_{K})\mbox{ for all }x\in \Sigma_{W}$$
Since $\theta$ is a homomorphism it follows readily that $\lambda_{K}$
is also a homomorphism, therefore a $1$-dimensional representation of
$\Sigma_{W}$.  Notice that $\lambda_{K}$ is completely determined by
its values on basis elements $x_{J}$, that
$\lambda(x_{J})=\theta(x_{J})(c_{K})=\chi_{J}(c_{K})$,
and these values of $\lambda_{K}$ comprise the
column of the matrix $M^c(W)$ indexed by $K$.
In particular, $\lambda_{K}|_{{\cal Z}_{W}}$
takes integer values and reducing these values modulo $p$ we shall
obtain the irreducible representations in a field of characteristic
$p$.  We already knew (Lemma \ref{rank}) that the $p$-rank of $M^c(W)$
was $s$ and so the above arguments have now proved:

\begin{lemma} \label{columns}
\begin{enumerate}
\item The columns of $M^c(W)$ define the irreducible representations of
$\Sigma_{W}$
\item The columns of $M^c(W)$ modulo $p$ define the irreducible
representations of
$\Sigma(W,p)$ and $M^c(W)$ modulo $p$ has precisely $s$ distinct columns.
\end{enumerate}
\end{lemma}

According to this lemma the set $E$ indexes the irreducible
representations of $\Sigma_{W}$.  We now select a subset $F\subseteq
E$  to index the irreducible representations of $\Sigma(W,p)$.  In
principle any subset  that indexes $s$ distinct columns of $M^{c}(W)
\mod p$ will suffice but we shall make a specific choice so that our
results are easier to state.  In $M^{c}(W)\mod p$ there are exactly $s$
non-zero rows (see the proof of Lemma \ref{columns}) and we let
$F\subseteq E$ index
this set of rows.  Since $M^{c}(W)\mod p$ is lower triangular of rank
$p$, $F$ also indexes a set of distinct columns of $M^{c}(W)\mod p$.
  We define a matrix
$D=(d_{KL})$ whose rows and columns are indexed by the members of $E$
and $F$ respectively. If $K\in E, L\in F$ then $d_{KL}=1$ if columns $K$
and $L$ of $M^c(W)$ are
equal modulo $p$, $d_{KL}=0$ otherwise.  By the previous lemma,
the sets $E$ and $F$ index
the irreducible representations of $\Sigma_{W}$ and $\Sigma(W,p)$
respectively and,
since $D$ determines the structure of each irreducible representation of
$\Sigma_{W}$ when reduced modulo $p$, we have

\begin{proposition} $D$ is the decomposition matrix of the algebra
$\Sigma_{W}$.
\end{proposition}

Of course, the decomposition matrix can be defined in a much more
general context.  Whenever we have a finite dimensional algebra where
reduction modulo $p$ makes sense we can let $\{\tau_{i}\}$ be its irreducible
representations in characteristic zero, $\{\upsilon_{j}\}$ its irreducible
representations in characteristic $p$, and define $d_{ij}$ to be the
multiplicity of $\upsilon_{j}$ as a composition factor of $\tau_{i}$ when
$\tau_{i}$ is reduced
modulo $p$.  In our case the situation is quite simple: as all the
irreducible representations in question are $1$-dimensional these
multiciplicities are either $0$ or $1$.

However, it is convenient to remain with the more general situation
for a little longer.  So let ${\cal E}$ be an algebraically closed complete
local field.  Then ${\cal E}$ is the field of fractions of a principal ideal
domain $U$, $U$ has a maximal ideal $P$, and ${\cal F}=U/P$ is a field of
prime characteristic $p$.  Of course, for descent algebras we have
been working over the rational field but, because their irreducible
representations are $1$-dimensional, we can extend to a larger field
without any significant changes.

Let $A$ be an associative algebra over ${\cal E}$ with an order ${\cal
D}\subset A$.  Then $\bar{\cal D}={\cal D}/P{\cal D}$ is an algebra
over the field ${\cal F}$.  Moreover, for every ${\cal D}$-module $M$,
$\bar{M}=M/PM$ is a $\bar{\cal D}$-module in a natural way.

Suppose that $P_{1},P_{2},\ldots,P_{r}$ are a full set of principal
indecomposable modules for ${\cal D}$ over ${\cal E}$ and that
$T_{1},T_{2},\ldots,T_{r}$ are the associated irreducible modules
(recall that $P_{i}$ has a unique maximal submodule $\rad(P_{i})$ and
that $T_{i}\cong P_{i}/\rad(P_{i})$).  The Cartan matrix of ${\cal
D}$ is an $r\times r$ matrix $C=(c_{ij})$ whose $(i,j)$ entry is the
number of times that $T_{j}$ occurs as a composition factor of
$P_{i}$.  It is known (see Theorem 11.10 in \cite{Bur})
that $c_{ij}$ is the intertwining number
$i(P_{j},P_{i})=\dim_{\cal E}\mbox{Hom}(P_{j},P_{i})$.

In an exactly analogous way let $Q_{1},Q_{2},\ldots,Q_{s}$ be a full
set of principal indecomposable modules for $\bar{\cal D}$ over the
field ${\cal F}$ with associated irreducible modules
$U_{1},U_{2},\ldots,U_{s}$ and let
$\tilde{c}_{ij}=i(Q_{j},Q_{i})=\dim_{\cal F}\mbox{Hom}(Q_{j},Q_{i})$ be the
Cartan matrix
$\tilde{C}$ of $\bar{\cal D}$.

The algebras ${\cal D}$ and $\bar{\cal D}$ are related by the
decomposition matrix $D$ which describes how each irreducible ${\cal
D}$-module $T_{i}$ behaves when reduced ``mod $p$''.  Specifically,
$D=(d_{ij})$
is an $r\times s$ matrix where $d_{ij}$ is the number of composition
factors of $\bar{T}_{i}$ which are isomorphic to $U_{j}$.  Again
by, Theorem 11.10 of \cite{Bur}, $d_{ij}=i(Q_{j},\bar{T}_{i})$.

By \cite{Bur} Theorem 37.4 there exist direct summands $R_{i}$ of the
regular ${\cal
D}$-module such that $\bar{R}_{j}=Q_{j}$ and we may write
$$R_{j}=\bigoplus_{k} h_{kj}P_{k}$$
where this equation signifies that $R_{j}$ can be expressed as a sum
of principal indecomposable modules, and that the module $P_{k}$
occurs as an isomorphism type $h_{kj}$ times.

At this point we note two applications of \cite{Bur} Lemma 38.1:
$$d_{ij}=i(Q_{j},\bar{T}_{i})=i(\bar{R}_{j},\bar{T}_{i})=i(R_{j},T_{i})$$
and
$$\tilde{c}_{ij}=i(Q_{j},Q_{i})=i(\bar{R}_{j},\bar{R}_{i})=i(R_{j},R_{i})$$

 From the first of these we have
\begin{eqnarray*}
d_{ij}&=&i(\bigoplus_{k} h_{kj}P_{k},T_{i})\\
&=&\sum_{k} h_{kj}i(P_{k},T_{i})\\
&=&h_{ij}
\end{eqnarray*}
since $i(P_{k},T_{i})=0$ unless $k=i$ and $i(P_{k},T_{k})=1$.

 From the second we obtain
\begin{eqnarray*}
\tilde{c}_{ij}&=&i(\bigoplus_{k} h_{kj}P_{k},\bigoplus_{l} h_{li}P_{l})\\
&=&\sum_{k,l} d_{kj}d_{li}i(P_{k},P_{l})\\
&=&\sum_{k,l} d_{kj}d_{li}c_{kl}
\end{eqnarray*}
In other words we have
\begin{theorem}
\label{C=DCD}
 $\tilde{C}=D^{T}CD$
\end{theorem}

A proof of this theorem in the language of Groethendieck groups has
recently been given in \cite{Geck} but it is likely that the result is
not new.  Clearly our proof has drawn heavily on the approach of
Burrow  \cite{Bur} who considered the case that $A$ was a group algebra
and where $C=I$ since group algebras are semi-simple.
It is plausible that Burrow knew Theorem
\ref{C=DCD} over 30 years ago.  Nevertheless the result deserves to be
better known and we make no further apology for including it.

We return now to the special case of descent algebras.  The
decomposition matrix of a descent algebra has been defined in terms of
a table of characters of the corresponding Coxeter group.  We have the
following easy result:

\begin{proposition}\label{preg} Let $K\in E,L\in F$ head columns of the matrix
$M^c(W)$.  Then, if $c_{K}$ and $c_{L}$ have conjugate $p$-regular parts,
$d_{KL}=1$.
\end{proposition}
\begin{proof}.  By the arguments in \S 82 of \cite{C&R} every
character $\chi_{J}$ takes equal values modulo $p$ on $c_{K}$ and
$c_{L}$.  Thus, $\lambda_{K}=\lambda_{L} \mod p$ and so $d_{KL}=1$.
\end{proof}.

In the remainder of this section we shall consider descent algebras
according to their Coxeter type.  By a combination of theoretical
argument and computer calculation we obtain a description of the
decomposition matrix in all cases and this shows that, often, the
converse of Proposition \ref{preg} is true.

We let $\pi(n)$ denote the number of partitions of $n$.  This
non-standard notation is necessary since we also define $\pi(n,p)$ as
the number of partitions of $n$ in which no part has multiplicity $p$
or more.  We note the following result (see \cite{james-kerber} p.41):

\begin{lemma}\label{partition-equality}
$\pi(n,p)$ is the number of partitions of $n$ into parts not divisible
by $p$.
\end{lemma}

\subsection{Representation Theory of $\Sigma(A_{n-1},p)$}

In this subsection we let $W=A_{n-1}$ which is best described as the
symmetric group
$S_{n}$ acting in the usuail way on $\{1,2,\ldots,n\}$
with generating set $S=\{(i,i+1)|i=1,\ldots,n-1\}$.  If
$K\subseteq S$ then the Coxeter element $c_{K}$ has cycles on sets
$[u..v]$ of consecutive integers.  The ordered list of cycle lengths
(one cycle appearing before another if it permutes integers with
smaller values) determines and is determined by $K$.  Therefore the
subsets of $S$ can be parameterised by compositions of $n$.  The
following lemma and corollary are easy consequences of this
parameterisation and Lemmas \ref{rank} and \ref{columns}

\begin{lemma}
\label{norm-Sn}
\begin{enumerate}
\item If $K,L\subseteq S$ then $K$ is conjugate to $L$ if and only if
the corresponding compositions determine the same partition of $n$.
\item If $K\subseteq S$ and its corresponding composition
has $a_{i}$ components equal to $i$ (for $i=1,\ldots,n$) then
$$[N(W_{K}):W_{K}]=a_{1}!a_{2}!\ldots a_{n}!$$
\end{enumerate}
\end{lemma}

\begin{corollary}\label{partitions}
\begin{enumerate}
\item $r=\pi(n)$
\item $s=\pi(n,p)$
\end{enumerate}
\end{corollary}

\begin{theorem}
\label{decompA}Let $W$ be one of the Coxeter groups $A_{n-1}$ and let
$K\in E$, $L\in F$.  Then
$d_{KL}=1$ if and only if $c_{K}$ and $c_{L}$ have conjugate
$p$-regular parts.
\end{theorem}
\begin{proof}  There are two equivalence relations $\rho_{1},\rho_{2}$
on the set $E$ (which indexes the columns of $M^c(W)$):
$$(K,J)\in \rho_{1} \mbox{ if } \lambda_{K}=\lambda_{J} \mod p$$
$$(K,J)\in \rho_{2} \mbox{ if } c_{K},c_{J} \mbox{ have conjugate }
p\mbox{-regular parts}$$
We have seen (Proposition \ref{preg}) that $\rho_{2}\subseteq \rho_{1}$.
However, the number of $\rho_{1}-$equivalence classes is $s$ (Lemma
\ref{columns}) and this is $\pi(n,p)$ (Corollary \ref{partitions}).
By Lemma \ref{partition-equality} this is also the number of partitions
with no part divisible by $p$ which is the number of equivalence
classes of $\rho_{2}$.  Hence $\rho_{1}=\rho_{2}$ and the theorem
follows.
\end{proof}

\subsection{Representation Theory of $\Sigma(B_{n},p)$}

It is convenient to represent $B_{n}$ as a permutation group on
$\{\pm 1,\ldots,\pm n\}$ with block system $\{i,-i\}_{i=1}^{n}$ on
which it acts as the full symmetric group with kernel of order
$2^{n}$.  The set of Coxeter generators
$S=\{s_{0},s_{1},\ldots,s_{n-1}\}$ is defined as $s_{0}=(-1,1)$ and
$s_{i}=(i,i+1)(-i,-i-1), 1\leq i\leq n-1$.

Let $K\subseteq S$ and consider the Coxeter element $c_{K}$.  If
$c_{K}$ has a cycle $(a,b,\ldots)$ consisting of positive elements (a
positive cycle)
then it will also have a corresponding negative cycle
$(-a,-b,\ldots)$.  Furthermore, at most one cycle of $c_{K}$ can
contain both positive and negative elements; such a cycle is present
if and only if $s_{0}\in K$.  We may write
\begin{equation}
c_{K}=x_{0}x_{1} \label {x0x1}
\end{equation}
where $x_{0}$ is the cycle containing both positive and negative
elements (or $x_{0}=1$ if there is no such cycle) and $x_{1}$ is the
product of all the other cycles (positive and negative in matching
pairs); note that $x_{0}$ commutes with $x_{1}$.  Each positive cycle is on
some range $[u..v]$ of consecutive
integers and the list of lengths of positive cycles taken in the
natural order (as in the previous subsection) determines and is
determined by $K$.  In this way the subsets of $S$ can be parameterised
by compositions of integers $m, 0\leq m\leq n$.  The following result
is a consequence of the results of \cite{Howlett}.

\begin{lemma}
\label{B-normaliser}
\begin{enumerate}
\item If $K,L\subseteq S$ then $K$ is conjugate to $L$ if and only if
the corresponding compositions determine the same partition.
\item If $K\subseteq S$ and the corresponding composition is a
composition of $m$ with $a_{i}$ components of size $i$ and $t$
components in all then
$$[N(W_{K}):W_{K}]=2^{t}a_{1}!a_{2}!\ldots a_{m}!$$
\end{enumerate}
\end{lemma}

\begin{corollary}\label{rsBn}
\begin{enumerate}
\item $r=\sum_{m=0}^{n}\pi(m)$
\item If $p\not = 2$ then $s=\sum_{m=0}^{n}\pi(m,p)$
\end{enumerate}
\end{corollary}

Let $K$ be one of the subsets indexing the rows and columns of $M^c(W)$
and $c_{K}=x_{0}x_{1}$ as in Equation \ref{x0x1}.  If $x_{1}$ is a $p$-regular
element we say that $K$ is a {\em $p$-special} subset of $S$.  Since the
order of $x_{1}$ is the lowest common multiple of its cycle lengths,
$K$ is $p$-special if and only if the partition corresponding to $K$ has
no part divisible by $p$.  By Lemma \ref{partition-equality} and Corollary
\ref{rsBn}, there are precisely $s$ $p$-special subsets when $p\neq 2$.

\begin{lemma}
\label{special}If $K\subseteq S$ there exists a $p$-special $K_{1}\subseteq
S$ such that $c_{K}$ and $c_{K_{1}}$ have conjugate $p$-regular parts.
\end{lemma}
\begin{proof}
Let $c_{K}=x_{0}x_{1}$ as in Equation \ref{x0x1} and
let $x_{2}$ be the $p$-regular part of $x_{1}$.  Since $x_{2}$ is a
power of $x_{1}$, its cycles also come in matching positive, negative
pairs.  Therefore $x_{2}$ is conjugate, via a permutation in the
centraliser of
$x_{0}$, to a Coxeter element $x_{3}$ with this property.
But then $x_{0}x_{3}$ is also a
Coxeter element $c_{K_{1}}$ whose $p$-regular part is conjugate to that
of $x_{0}x_{1}$.
\end{proof}

\begin{lemma}
\label{indexing} If $p\neq 2$ the columns of $M^c(W)$ which are indexed by
the $p$-special
subsets provide a full set of irreducible representations of
$\Sigma(B_{n},p)$.
\end{lemma}
\begin{proof}
By the last lemma the columns of $M^c(W) \mod p$ indexed by $p$-special subsets
contain a full set of distinct columns and since there are $s$ such
columns they must yield a complete set of irreducible representations of
$\Sigma(B_{n},p)$.
\end{proof}

\begin{theorem}
\label{decompB}Let $W$ be one of the Coxeter groups $B_{n}$ and let
$K\in E,L\in F$. If $p\neq
2$ then $d_{KL}=1$ if and only if $c_{K}$ and $c_{L}$ have conjugate
$p$-regular parts.  If $p=2$ then $F=\{S\}$ and $d_{KS}=1$ for all $K$.
\end{theorem}
\begin{proof}
Suppose first that $p\neq 2$.
Proposition   \ref{preg} has proved one implication already.  For the other,
suppose $d_{KL}=1$ and let $K_{1},L_{1}$ be the $p$-special subsets,
guaranteed by
Lemma \ref{special}, such that $K,K_{1}$ have conjugate $p$-regular
parts and $L,L_{1}$ have conjugate $p$-regular parts.  Then, by
Proposition \ref{preg},  $d_{K_{1}L_{1}}=1$ and Lemma \ref{indexing} shows
that $K_{1}=L_{1}$.

If $p=2$, Lemma \ref{B-normaliser}
implies that the only $K\in E$ for which $2$ does not divide
$[N(W_{K}):W_{K}]$ is the
one with $t=0$, namely $K=S$.  Therefore $\Sigma(B_{n},2)$ has just
one irreducible representation and so $d_{KL}=1$ for all $K\in E$,
$L\in F=\{S\}$.
\end{proof}

\subsection{Representation Theory of $\Sigma(D_n, p)$.}

The Coxeter group $(W, S)$ of type $D_n$ can be considered as a normal
subgroup of index  $2$ in  the Coxeter  group $(\hat{W}, \hat{S})$  of
type $B_n$.  As such  its set of Coxeter generators  is $S = \{u, s_1,
\dots, s_{n-1}\}$  where,  as  in the previous subsection,
$\hat{S}  =  \{s_0, s_1,
\dots, s_{n-1}\}$ and $u = s_0 s_1 s_0  = (-1, 1)(1, 2)(-1, -2)(-1, 1)
= (-1, 2)(1, -2)$.

For any $K \subseteq S$ the parabolic subgroup  $W_K$ is isomorphic to
$W_0 \times W_1$ where $W_0$  is of type  $D_{n_0}$ for some $n_0 \leq
n$,  $n_0 \not=1$  and $W_1  =  \Group{K_1}$  for some $K_1  \subseteq
\{s_{n_0},  \dots,   s_{n-1}\}$.  (Here the group    of  type $D_2$ is
$\Group{u, s_1}$ and  isomorphic to a group of  type  $A_1 \times A_1$
and the group of type $D_3$ is $\Group{u, s_1, s_2}$ and isomorphic to
a group of type $A_3$.)   If  $n_0 = 0$   then $W_1$ is a subgroup  of
either the group $W'$ generated by $S' = \{s_1, s_2, \dots, s_{n-1}\}$
or   the  group $W''$  generated  by   $S'' = \{u,    s_2, s_3, \dots,
s_{n-1}\}$ which are both of type $A_{n-1}$.

Thus to each subset  $K \subseteq S$  there is associated via $W_1$  a
composition of $m  \leq n$.  Each  composition occurs this way, except
those   of $n-1$.  Conversely, for  each  composition  $\lambda$ of $m
\not= n-1$, there is a unique  $K \subseteq S$,  unless $\lambda$ is a
composition of $n$ with $\lambda_1  > 1$.  In that  case there are two
subsets with that label, each containing exactly one of $s_1$ and $u$.

Consider $K, L \subseteq S$.  Then $K$ and $L$ are conjugate in $W$ if
and only if  their   corresponding  compositions determine  the   same
partition, unless that partition is a partition of  $n$ with all parts
even.  In that case $K$ and $L$ are conjugate only if they both lie in
$S'$ or both in $S''$.

Consider $W_K = W_0 \times W_1$ with $W_0$  of type $D_{n_0}$ for some
$n_0 \geq  2$.    Then there is  a  parabolic   subgroup $\hat{W}_K  =
\hat{W}_0  \times W_1$  of   $\hat{W}$ where  $\hat{W}_0$ is of   type
$B_{n_0}$.  We have $W_K = \hat{W}_K \cap W$ and $[\hat{W}_K:W_K]
=   2$.   Also  $[N_{\hat{W}}(\hat{W}_K):N_W(W_K)] =   2$  whence
$\beta_{KK}$ is computed  from the partition  corresponding to  $K$ in
the same way as in case $B_n$.

Now  let   $n_0 = 0$   and  let  $W_K$ be   a  subgroup of  $W'$  with
corresponding partition $\mu$.  Then $W_K$ is  a parabolic subgroup of
both $W$ and $\hat{W}$.  We have $N_{\hat{W}}(W_K) \subseteq W$ if and
only if   all parts of  $\mu$  are  even.  We   thus get the following
formula.

\begin{lemma} \label{la:D}
  Let $K \subseteq S$  with  corresponding partition $\mu =  (1^{m_1},
  2^{m_2}, \dots, n^{m_n})$.  Then
  \[
    [N_W(W_K) : W_K] = 2^{m_1}m_1!\, \dotsm \, 2^{m_n}m_n!\, a
  \]
  where $a  = 1$ unless $\mu$  is a partition of  $n$ and has at least
  one odd part.  In that case $a = 1/2$.
\end{lemma}

Let $c_K$ be  a Coxeter element  of $W_K$.   Again,  we have  a unique
decomposition $c_K = x_0 x_1$ where $x_0  \in W_0$ and  $x_1 \in W_1$.
We call $K$ a $p$-special subset if $x_1$ is $p$-regular.  And, by the
same argument as for type $B_n$, we have that for each $K \subseteq S$
there is a $p$-special $K_{1} \subseteq S$ such  that $c_K$ and
$c_{K_1}$ have
conjugate $p$-regular parts.

Similar considerations as  for type $B_n$ then  lead  to the following
description of the decomposition matrix for type $D_n$.

\begin{theorem}\label{decompD}
  Let $(W, S)$ be of type $D_n$ and  let $K \in E$, $L  \in F$.  If $p
  \not=2$ then  $d_{KL}  = 1$ if  and  only if   $c_K$ and  $c_L$ have
  conjugate $p$-regular parts.  If $p= 2$ and $n$ is even then we have
  $F = \{S\}$ and $d_{KS} = 1$ for all $K \in E$; if $n$ is odd then we
  have $F = \{S', S\}$ and $d_{KL} = 1$ if and only if  either $L = S$
  and $K \not= S'$ or $L = S'$ and $K = S'$.
\end{theorem}

\begin{proof}
  The theorem for $p \not=2$ follows as in case $B_n$.   For $p =2$ we
  show that either
  $\beta_{KL} = 0  \mod 2$ for all $K,  L \in E$ unless $K =
  S$, or $n$ is odd and $K=  L  = S'$.   Note that  by Lemma~\ref{la:D},
  $\beta_{KK}  = [N_W(W_K) :  W_K]$ is odd  only if all $m_i = 0$
  (whence $\mu$ is the  empty partition corresponding to  $K = S$) or,
  if $\mu$  is a partition  of $n$ with at  least one odd part  and at
  most  one $m_i =  1$ (whence $n$ is  odd and $\mu$  is the partition
  $[n]$ corresponding to $K = S'$).  Thus, for $K \not= S$, $\beta_{KL}$
  is even unless $n$ is odd and $K= S'$.

  Finally, in order to see that $\beta_{KL}$ is  even in the remaining
  cases (where $L = S'$ and $L$ not  conjugate to $K$) we consider the
  following action on complementary   pairs, first as a  $B_n$ action.
  Let $I = \{1, \dots, n\}$ and let
  \[
    X = \left\{\{P, Q\} \bigm| P, Q \subseteq I; P \cup Q = I;
      P \cap Q = \emptyset \right\}
  \]
  (so we always have $Q = I \setminus P$).   Then $B_n$ acts on $X$ as
  follows.   The  action of
  $s_i$ ($i \geq 1$) is  induced from its action as  $(i, i+1)$ on $I$
  and the action of $s_0$ is given by
  \[
    \{P, Q\}^{s_0} = \{P \perp \{1\},     Q \perp \{1\}\},
  \]
  where  $A \perp B = (A  \setminus B) \cup   (B \setminus A)$ denotes the
  symmetric difference of the  sets $A, B$.   Note that, if  we define
  $t_i = s_i \dotsm s_1  s_0 s_1 \dotsm   s_i$ then $t_i$ acts  as
  \[
    \{P, Q\}^{t_i} = \{P \perp \{i{+}1\}, Q \perp \{i{+}1\}\},
  \]
  the  symmetric difference with  $\{i{+}1\}$, and the longest element
  $w_0 =   t_0 t_1  \dotsm  t_{n-1}$ of  $\hat{W}$  acts as  symmetric
  difference with $I$   whence it  fixes   every point   in $X$.   The
  complementary  pair $\{P,   Q\}$ arises from   $\{\emptyset, I\}$ by
  taking symmetric differences with $P$  (or $Q$).  Thus the action of
  $\hat{W}$ is transitive on all of  the $2^{n-1}$ complementary pairs
  in $X$  and the stabiliser  of  $\{\emptyset, I\}$  is  $\Group{s_1,
    \dots, s_{n-1}, w_0}$, a group of index $2^{n-1}$ in $\hat{W}$.

  Now let $n$ be odd and restrict the action to $W$.  Then, since $w_0
  \not \in W$,  the stabiliser of $\{\emptyset,  I\}$ in  $W$ is $W'$,
  which is of index $2^{n-1}$ in  $W$.  Hence $W$ acts transitively on
  $X$ and the action   is equivalent to  the action  on the cosets  of
  $W'$, the one we are interested in.

  We know that $\beta_{KL} = 0$  whenever $W_L$ is  not conjugate to a
  subgroup of $W_K$.   It remains to investigate  the  fixed points of
  parabolic  subgroups of $W'$  which is of  type $A_{n-1}$.  Consider
  $s_1$ and its  fixed points.  If $n  > 2$ then  $\{P, Q\}$ is stable
  under $s_1$  if and only  if  $\{1, 2\} \subseteq   P$ or $\{1,  2\}
  \subseteq  Q$.  In   either case taking  symmetric  differences with
  $\{1, 2\}$ yields a different point $\{P', Q'\}$ which is also fixed
  by $s_1$.  So the fixed points of $s_1$ come in pairs.

  A  similar  argument applies  to  a Coxeter   element  $c_L$ of  any
  parabolic subgroup $W_L$ of $W'$ unless $L  =S'$.  Here we denote by
  $J \subseteq I$ the set of points moved  by $c_L$. Then we find that
  $\{P, Q\}$ is  stable under $c_L$ if and  only if $J \subseteq P$ or
  $J \subseteq Q$.   Again,   taking symmetric differences   with  $J$
  produces  a  different fixed point  $\{P',  Q'\}$.  This  shows that
  $\beta_{KL}$  is  even for all  proper  parabolic subgroups $W_L$ of
  $W'$.
\end{proof}

\subsection{Representation Theory of Exceptional Types}

The   descriptions   of the decomposition matrices      in the case of
the classical types in the previous subsections
are special cases  of a more general classification of
columns of the  parabolic table of  marks that are  equal if taken mod
$p$.

For  this more general classification we need to extend the notion of
having the same $p$-regular part.  Let $w'$ be the $p$-regular  part
of  $w \in  W$ and let
$\rightarrow_p$ be the relation on $E$  defined by $J \rightarrow_p K$
if $\left<w'\right>^c$ is  conjugate to $W_K$ for some  $w \in  $ such
that $\left<w\right>^c$ is conjugate to $W_J$.

\begin{theorem}
\label{decompGeneral}
  Let $K \in E$ and $L \in F$.  Then $d_{KL}  = 1$ if  and only if $K$
  and $L$  lie  in the  same  class of  the  equivalence generated  by
  $\rightarrow_p$.
\end{theorem}

The  following tables, which we have computed using the CHEVIE
\cite{chevie}
package in GAP \cite{GAP},  describe  the  decomposition matrices for   the
exceptional types.  The proof of the  theorem follows by inspection of
these tables and the parabolic tables of marks reduced mod $p$,
together with theorems \ref{decompA}, \ref{decompB} and \ref{decompD}.
Note that the theorem
is also true for the dihedral types $I_{2}(m)$ (see
\cite{Steph-thesis} for a full account of the representation theory in
all characteristics in this case).

In each case we give for any $K \in E$ and  for any prime $p$ dividing
the order of $W$ the list of $L \in E$ such that  $K \rightarrow_p L$.
The first entry in each list is determined  by the $p$-regular part of
a  Coxeter element,  and   the   number in  parenthesis  denotes   the
representative  in  the  equivalence obtained as    the closure of the
relation $\rightarrow_p$  if  different from the  first   entry of the
list.  If the list for  $K$ consists of $K$ only  and this is also the
representative  we  just  have   a dot (.)     as entry.   Note   that
conversely, all the representatives have a dot entry.

For  each $K \in E$  we also list  its isomorphism type, possibly with
dashes ($'$ and  $''$) to distinguish isomorphic parabolic  subgroups,
and the index $\beta_{KK}$ of $W_{K}$ in its normalizer in $W$.

%%  Table generated by GAP function 'LatexPTP'.  %%%%%%
\begin{table}[htbp]
\begin{center}\begin{tabular}{rc|r|ccc}
&&$\beta_{KK}$&$p = 2$&$p = 3$&$p = 5$\\\hline
1&$1$&$51840$&.&.&.\\
2&$A_1$&$720$&1&.&.\\
3&$A_1{\times}A_1$&$48$&1&.&.\\
4&$A_2$&$72$&4 (1)&1&.\\
5&$A_1{\times}A_1{\times}A_1$&$12$&1&.&.\\
6&$A_2{\times}A_1$&$6$&4 (1)&2&.\\
7&$A_3$&$8$&1&.&.\\
8&$A_2{\times}A_1{\times}A_1$&$2$&4 (1)&3&.\\
9&$A_2{\times}A_2$&$12$&.&1&.\\
10&$A_3{\times}A_1$&$2$&1&.&.\\
11&$A_4$&$2$&.&.&1\\
12&$D_4$&$6$&4, 1 (1)&12 (1)&.\\
13&$A_2{\times}A_2{\times}A_1$&$2$&9&2&.\\
14&$A_4{\times}A_1$&$1$&11&.&2\\
15&$A_5$&$2$&9&5&.\\
16&$D_5$&$1$&1, 4&.&.\\
17&$E_6$&$1$&17, 9 (9)&12, 1 (1)&.\\
\end{tabular}\end{center}
\caption{Decomposition matrix for $E_6$.}
\end{table}

%%  Table generated by GAP function 'LatexPTP'.  %%%%%%
\begin{table}[htbp]
\begin{center}\begin{tabular}{rc|r|cccc}
&&$\beta_{KK}$&$p = 2$&$p = 3$&$p = 5$&$p = 7$\\\hline
1&$1$&$2903040$&.&.&.&.\\
2&$A_1$&$23040$&1&.&.&.\\
3&$A_1{\times}A_1$&$768$&1&.&.&.\\
4&$A_2$&$1440$&4 (1)&1&.&.\\
5&$(A_1{\times}A_1{\times}A_1)'$&$1152$&1&.&.&.\\
6&$(A_1{\times}A_1{\times}A_1)''$&$96$&1&.&.&.\\
7&$A_2{\times}A_1$&$48$&4 (1)&2&.&.\\
8&$A_3$&$96$&1&.&.&.\\
9&$A_1{\times}A_1{\times}A_1{\times}A_1$&$48$&1&.&.&.\\
10&$A_2{\times}A_1{\times}A_1$&$8$&4 (1)&3&.&.\\
11&$A_2{\times}A_2$&$24$&11 (1)&1&.&.\\
12&$(A_3{\times}A_1)'$&$48$&1&.&.&.\\
13&$(A_3{\times}A_1)''$&$8$&1&.&.&.\\
14&$A_4$&$12$&14 (1)&.&1&.\\
15&$D_4$&$48$&4, 1 (1)&15 (1)&.&.\\
16&$A_2{\times}A_1{\times}A_1{\times}A_1$&$12$&4 (1)&5&.&.\\
17&$A_2{\times}A_2{\times}A_1$&$4$&11 (1)&2&.&.\\
18&$A_3{\times}A_1{\times}A_1$&$4$&1&.&.&.\\
19&$A_3{\times}A_2$&$4$&4 (1)&8&.&.\\
20&$A_4{\times}A_1$&$2$&14 (1)&.&2&.\\
21&$D_4{\times}A_1$&$8$&4, 1 (1)&.&.&.\\
22&$A_5'$&$12$&11 (1)&5&.&.\\
23&$A_5''$&$4$&11 (1)&6&.&.\\
24&$D_5$&$4$&1, 4&.&.&.\\
25&$A_3{\times}A_2{\times}A_1$&$2$&4 (1)&12&.&.\\
26&$A_4{\times}A_2$&$2$&26 (1)&14&4&.\\
27&$A_5{\times}A_1$&$2$&11 (1)&9&.&.\\
28&$D_5{\times}A_1$&$2$&1, 4&.&.&.\\
29&$A_6$&$2$&29 (1)&.&.&1\\
30&$D_6$&$2$&14, 1, 4, 11 (1)&.&.&.\\
31&$E_6$&$2$&31, 11 (1)&15, 1 (1)&.&.\\
32&$E_7$&$1$&31, 1, 4, 11, 14, 26, 29 (1)&32, 5 (5)&.&.\\
\end{tabular}\end{center}
\caption{Decomposition matrix for $E_7$.}
\end{table}

%%  Table generated by GAP function 'LatexPTP'.  %%%%%%
\begin{table}[htbp]  %\tabcolsep=2pt
\begin{center}\begin{tabular}{rc|r|cccc}
&&$\beta_{KK}$&$p = 2$&$p = 3$&$p = 5$&$p = 7$\\\hline
1&$1$&$696729600$&.&.&.&.\\
2&$A_1$&$2903040$&1&.&.&.\\
3&$A_1{\times}A_1$&$46080$&1&.&.&.\\
4&$A_2$&$103680$&4 (1)&1&.&.\\
5&$A_1{\times}A_1{\times}A_1$&$2304$&1&.&.&.\\
6&$A_2{\times}A_1$&$1440$&4 (1)&2&.&.\\
7&$A_3$&$3840$&1&.&.&.\\
8&$A_1{\times}A_1{\times}A_1{\times}A_1$&$384$&1&.&.&.\\
9&$A_2{\times}A_1{\times}A_1$&$96$&4 (1)&3&.&.\\
10&$A_2{\times}A_2$&$288$&10 (1)&1&.&.\\
11&$A_3{\times}A_1$&$96$&1&.&.&.\\
12&$A_4$&$240$&12 (1)&.&1&.\\
13&$D_4$&$1152$&4, 1 (1)&13 (1)&.&.\\
14&$A_2{\times}A_1{\times}A_1{\times}A_1$&$24$&4 (1)&5&.&.\\
15&$A_2{\times}A_2{\times}A_1$&$24$&10 (1)&2&.&.\\
16&$A_3{\times}A_1{\times}A_1$&$16$&1&.&.&.\\
17&$A_3{\times}A_2$&$16$&4 (1)&7&.&.\\
18&$A_4{\times}A_1$&$12$&12 (1)&.&2&.\\
19&$D_4{\times}A_1$&$48$&4, 1 (1)&19 (2)&.&.\\
20&$A_5$&$24$&10 (1)&5&.&.\\
21&$D_5$&$48$&1, 4&.&.&.\\
22&$A_2{\times}A_2{\times}A_1{\times}A_1$&$8$&10 (1)&3&.&.\\
23&$A_3{\times}A_2{\times}A_1$&$4$&4 (1)&11&.&.\\
24&$A_4{\times}A_1{\times}A_1$&$4$&12 (1)&.&3&.\\
25&$A_3{\times}A_3$&$8$&1&.&.&.\\
26&$A_4{\times}A_2$&$4$&26 (1)&12&4&.\\
27&$D_4{\times}A_2$&$12$&10, 4 (1)&13 (1)&.&.\\
28&$A_5{\times}A_1$&$4$&10 (1)&8&.&.\\
29&$D_5{\times}A_1$&$4$&1, 4&.&.&.\\
30&$A_6$&$4$&30 (1)&.&.&1\\
31&$D_6$&$8$&12, 1, 4, 10 (1)&.&.&.\\
32&$E_6$&$12$&32, 10 (1)&13, 1 (1)&.&.\\
33&$A_4{\times}A_2{\times}A_1$&$2$&26 (1)&18&6&.\\
34&$A_4{\times}A_3$&$2$&12 (1)&.&7&.\\
35&$A_6{\times}A_1$&$2$&30 (1)&.&.&2\\
36&$D_5{\times}A_2$&$2$&4, 10 (1)&21&.&.\\
37&$A_7$&$2$&1&.&.&.\\
38&$E_6{\times}A_1$&$2$&32, 10 (1)&19, 2 (2)&.&.\\
39&$D_7$&$2$&10, 1, 4, 12 (1)&.&.&.\\
40&$E_7$&$2$&32, 1, 4, 10, 12, 26, 30 (1)&40, 5 (5)&.&.\\
41&$E_8$&$1$&41, 1, 4, 10, 12, 26, 30, 32 (1)&41, 1, 13 (1)&41, 1 (1)&.\\
\end{tabular}\end{center}
\caption{Decomposition matrix for $E_8$.}
\end{table}

%%  Table generated by GAP function 'LatexPTP'.  %%%%%%
\begin{table}[htbp]
\begin{center}\begin{tabular}{rc|r|cc}
&&$\beta_{KK}$&$p = 2$&$p = 3$\\\hline
1&$1$&$1152$&.&.\\
2&$A_1'$&$48$&1&.\\
3&$A_1''$&$48$&1&.\\
4&$A_1{\times}A_1$&$4$&1&.\\
5&$A_2'$&$12$&5 (1)&1\\
6&$A_2''$&$12$&6 (1)&1\\
7&$B_2$&$8$&1&.\\
8&$(A_2{\times}A_1)'$&$2$&5 (1)&2\\
9&$(A_2{\times}A_1)''$&$2$&6 (1)&3\\
10&$B_3'$&$2$&5, 1 (1)&.\\
11&$B_3''$&$2$&6, 1 (1)&.\\
12&$F_4$&$1$&12, 1, 5, 6 (1)&12, 1 (1)\\
\end{tabular}\end{center}
\caption{Decomposition matrix for $F_4$.}
\end{table}

%%  Table generated by GAP function 'LatexPTP'.  %%%%%%
\begin{table}[htbp]
\begin{center}\begin{tabular}{rc|r|ccc}
&&$\beta_{KK}$&$p = 2$&$p = 3$&$p = 5$\\\hline
1&$1$&$120$&.&.&.\\
2&$A_1$&$4$&1&.&.\\
3&$A_1{\times}A_1$&$2$&1&.&.\\
4&$A_2$&$2$&4 (1)&1&.\\
5&$I_2(5)$&$2$&5 (1)&.&1\\
6&$H_3$&$1$&5, 1, 4 (1)&.&.\\
\end{tabular}\end{center}
\caption{Decomposition matrix for $H_3$.}
\end{table}

%%  Table generated by GAP function 'LatexPTP'.  %%%%%%
\begin{table}[htbp]
\begin{center}\begin{tabular}{rc|r|ccc}
&&$\beta_{KK}$&$p = 2$&$p = 3$&$p = 5$\\\hline
1&$1$&$14400$&.&.&.\\
2&$A_1$&$120$&1&.&.\\
3&$A_1{\times}A_1$&$8$&1&.&.\\
4&$A_2$&$12$&4 (1)&1&.\\
5&$I_2(5)$&$20$&5 (1)&.&1\\
6&$A_2{\times}A_1$&$2$&4 (1)&2&.\\
7&$I_2(5){\times}A_1$&$2$&5 (1)&.&2\\
8&$A_3$&$2$&1&.&.\\
9&$H_3$&$2$&5, 1, 4 (1)&.&.\\
10&$H_4$&$1$&10, 1, 4, 5 (1)&10, 1 (1)&10, 1 (1)\\
\end{tabular}\end{center}
\caption{Decomposition matrix for $H_4$.}
\end{table}

\subsection {Cartan matrices}

By Theorems \ref{C=DCD} and  \ref{decompGeneral}  the Cartan matrix of
$\Sigma(W,p)$ can be determined once it is known for $\Sigma_{W}$.
Types $A$ and $B$ can therefore be handled by Theorem 5.4
of \cite{G&R} and Theorem 3.3 of \cite{Ber}
which give the Cartan matrices in characteristic
zero.  Furthermore, the work of \cite{Steph-thesis} allows the
dihedral case to be solved.  However, we have not calculated the Cartan
matrix in characteristic zero in any other cases; such a calculation awaits
a more detailed study of these algebras.

\end{document}